\documentclass[12pt,letterpaper,english]{amsart}
\usepackage{lmodern}
\usepackage{helvet}
\usepackage{hyperref}

\usepackage[T1]{fontenc}
\usepackage[latin9]{inputenc}
\usepackage{mathrsfs}
\usepackage{amsthm}
\usepackage{amstext}
\usepackage{amssymb}
\usepackage{tikz}
\usepackage{esint}
\usepackage{graphicx}
\usepackage{cancel}
\usepackage{enumitem}
\DeclareFontFamily{OT1}{pzc}{}
\DeclareFontShape{OT1}{pzc}{m}{it}{<-> s * [1.10] pzcmi7t}{}
\DeclareMathAlphabet{\mathpzc}{OT1}{pzc}{m}{it}
\usepackage[bbgreekl]{mathbbol}
\DeclareSymbolFontAlphabet{\mathbb}{AMSb}
\DeclareSymbolFontAlphabet{\mathbbl}{bbold}

\makeatletter

\pdfpageheight\paperheight
\pdfpagewidth\paperwidth

\newcommand{\noun}[1]{\textsc{#1}}

\numberwithin{equation}{section}
\numberwithin{figure}{section}
\usepackage{enumitem}		% customizable list environments
\theoremstyle{plain}
\newtheorem{thm}[equation]{\protect\theoremname}
\theoremstyle{remark}
\newtheorem{rem}[equation]{\protect\remarkname}
\theoremstyle{definition}

\theoremstyle{definition}
\newtheorem{example}[equation]{\protect\examplename}
\theoremstyle{plain}

\theoremstyle{plain}
\newtheorem{prop}[equation]{Proposition}
\theoremstyle{plain}

\usepackage{amsfonts}
\usepackage{amssymb}
\usepackage[all]{xy}
\usepackage{array}
\usepackage{lmodern}
\usepackage[T1]{fontenc}
\usepackage{bm}

\def\QQ{\mathbb{Q}}
\def\RR{\mathbb{R}}
\def\CC{\mathbb{C}}

\def\ZZ{\mathbb{Z}}
\def\PP{\mathbb{P}}
\def\DD{\mathbb{D}}
\def\GG{\mathbb{G}}
\def\EE{\mathbb{E}}

\def\W{\mathcal{W}}
\def\VV{\mathbb{V}}
\def\F{\mathcal{F}}
\def\Z{\mathcal{Z}}
\def\V{\mathcal{V}}

\def\X{\mathcal{X}}
\def\XB{\overline{\X}}

\def\H{\mathcal{H}}
\def\HH{\mathbb{H}}

\def\D{\mathcal{D}}
\def\vf{\varphi}

\def\Res{\text{Res}}

\def\IH{\mathrm{IH}}

\def\gr{\mathrm{Gr}}

\def\E{\mathcal{E}}

\def\QB{\overline{\QQ}}
\def\CH{\mathrm{CH}}

\def\NF{\mathrm{NF}}
\def\ANF{\mathrm{ANF}}

\def\Hg{\mathrm{Hg}}
\def\ay{\mathbf{i}}

\def\fb{\mathfrak{b}}
\def\fa{\mathfrak{a}}
\def\BH{\mathbf{H}}
\def\U{\mathcal{U}}
\def\UAN{\U_{\CC}^{\text{an}}}

\def\BQ{\mathbf{Q}}
\def\BE{\mathbf{E}}
\def\b{\bullet}
\def\cO{\mathcal{O}}
\def\uh{\underline{h}}
\def\HMS{\mathrm{Hom}_{\mathrm{MHS}}}
\def\EMS{\mathrm{Ext}^1_{\mathrm{MHS}}}
\def\AJ{\mathrm{AJ}}
\def\Z{\mathcal{Z}}
\def\fV{\mathfrak{V}}
\def\J{\mathcal{J}}
\def\pb{\overline{\pi}}
\def\kmath{\mathsf{k}}
\def\UR{\mathfrak{Ur}}
\def\SC{\mathscr{C}}
\def\SF{\mathscr{F}}
\def\SJ{\mathscr{J}}
\def\BV{\mathbf{V}}
\def\VE{\varepsilon}
\def\SE{\mathsf{e}}
\def\ff{\mathfrak{f}}
\def\FV{\mathfrak{V}}
\def\fz{\mathfrak{z}}
\def\RI{\mathrm{I}}
\def\RII{\mathrm{II}}
\def\RIII{\mathrm{III}}
\def\RIV{\mathrm{IV}}
\def\BP{\mathbf{P}}
\def\SH{\mathsf{H}}

\theoremstyle{definition}
\newtheorem*{thx}{Acknowledgments}

\makeatother

\usepackage{babel}
  \providecommand{\corollaryname}{Corollary}
  \providecommand{\definitionname}{Definition}
  \providecommand{\remarkname}{Remark}
\providecommand{\theoremname}{Theorem}
 \providecommand{\examplename}{Example}

\begin{document}

\title[Hypergeometric normal functions]{Three vignettes on hypergeometric normal functions}

%\author{}

\author[M.~Kerr]{Matt Kerr}

\subjclass[2000]{14C30, 19E15, 33C20}
\begin{abstract}
We use Hodge-theoretic methods to (i) explain number-theoretic identities of a type recently considered by Guillera and Zudilin, (ii) describe the Frobenius dual of Abel-Jacobi period functions, and (iii) offer a new proof of Golyshev's conjecture on algebraic hypergeometrics, aided by an argument in the spirit of Lefschetz's (1,1) theorem.
\end{abstract}
\maketitle

\section{Introduction}\label{S1}

In his final published mathematical work, Poincar\'e studied holomorphic sections of a Jacobian bundle associated to a pencil of algebraic curves \cite{Po}. When these sections behave ``normally'' (with ``admissible'' logarithmic singularities) along the discriminant locus, he showed that they arise from an algebraic 1-cycle on the total space. Lefschetz's subsequent discovery, that fibering out a rational (1,1)-class on a smooth projective surface produces such a \emph{normal function} \cite{Le}, thus verified an instance of the Hodge Conjecture in the days before the Hodge decomposition was known.

More generally, for a smooth projective morphism $\pi\colon \mathcal{X}\to\U$ of smooth quasi-projective varieties, there is a correspondence between Hodge classes on $\mathcal{X}_{\CC}^{\text{an}}$ and admissible normal functions over $\U_{\CC}^{\text{an}}$ valued in generalized Jacobian bundles associated to $\pi$.  The algebraic $K$-groups of $\mathcal{X}$ provide a natural source of such Hodge classes, and according to Beilinson's generalization of the Hodge Conjecture, produce all of them when $\mathcal{X}$ is defined over $\QB$. We briefly review these notions in \S\ref{S2} and go into a bit more depth in \S\ref{S5}.

Among the more explicitly computable examples of normal functions are those taking values in Jacobians of hypergeometric variations of Hodge structure.  As recalled in \S\ref{S3}, these are variations over $\PP^1\setminus\{0,1,\infty\}$ whose Picard-Fuchs equations take the form
\begin{equation}\label{eq1.1}
\textstyle\mspace{50mu}L_{\fa,\fb}=\prod_j(D+\fb_j-1)-z\prod_j(D+\fa_j)\mspace{30mu}(D:=z\tfrac{d}{dz}),
\end{equation}
and which thus have (1-variable) hypergeometric functions as periods.  Ramanujan's identities and functional equations involving hypergeometrics count with his most significant discoveries.

The heart of this paper (in \S\S\ref{S4},\ref{S6},\ref{S7}) is a collection of three mathematical vignettes (or sketches) illustrating connections between ``Lefschetz'' and ``Ramanujan''.  That is, we want to examine how the basic relationship observed by Lefschetz, between normal functions in a family and Hodge classes on the total space, might shed new light on hypergeometric identities and function theory.  Each of the vignettes was motivated by a simple question.  For the first, recall the Pochhammer symbol $[a]_n:=a(a+1)\cdots(a+n-1)$.\vspace{2mm}

\textbf{Q1}: Where do ``upside-down Ramanujan identities'' like
\begin{equation}\label{eq1.2}
\sum_{n\geq 1}\frac{10n^2-6n+1}{n^5}\cdot \frac{n!^5(-\frac{1}{4})^n}{[\frac{1}{2}]_n^5}=-28\zeta(3)
\end{equation}
(due to Guillera \cite{Gu}) and
\begin{equation}\label{eq1.3}
\sum_{n\geq 1}\tfrac{5532n^4-5600n^3+2275n^2-425n+30}{n^9}\cdot\frac{n!^9(-\frac{1024}{3125})^n}{[\frac{1}{2}]_n^5[\frac{1}{5}]_n[\frac{2}{5}]_n[\frac{3}{5}]_n[\frac{4}{5}]_n}=-380928\zeta(5)
\end{equation}
(due to K-C Au \cite{Au}) \emph{really} come from?\vspace{2mm}

\textbf{Q2}: ``Most'' hypergeometric functions are transcendental.  But once you have one like $$f=\sum_{n\geq 0}\frac{(30n)!n!}{(15n)!(10n)!(6n)!}z^n$$ which is algebraic over $\QQ(z)$, it turns out that $$\exp\left(\int f\frac{dz}{z}\right)$$ is algebraic too!  Why?\vspace{2mm}

\textbf{Q3}: For variations of Hodge structure (VHS) of Calabi-Yau type (top Hodge number 1), The Poincar\'e pairing of a normal function $\nu$ with a section $\mu$ of the Hodge line bundle yields a (multivalued) \emph{quasi-period} function $\fV$ solving an \emph{inhomogeneous} Picard-Fuchs equation $L\fV=f$.  Dually (see \S\ref{S7}), there exists a choice of (multivalued) lift $\tilde{\nu}$ to the VHS such that\footnote{Here $D$ means $\nabla_D$, where $\nabla$ is the Gauss-Manin connection.} $D\tilde{\nu}$ is a constant multiple of $f\mu$.  In the case of a family of elliptic curves, $\fV$ is the usual Abel-Jacobi integral attached to a family of divisors.  Does the special lift $\tilde{\nu}$ also have a geometric interpretation, and can we compute it for hypergeometric families?

\begin{thx}
MK was supported by NSF Grant DMS-2502708 and the Simons Foundation, and is grateful to V.~Golyshev for numerous discussions out of which these ideas crystallized. He thanks D.~Akman, C.~Doran, M.~Elmi, I.~Gaiur, A.~Thompson, and W.~Zudilin for helpful conversations, and J.~Voight for a wonderful talk at Washington University based on \cite{DPVZ} which stimulated the first vignette.
\end{thx}

\section{Normal functions}\label{S2}

We fix some conventions for the duration.  Given a smooth projective morphism $\pi\colon \X\to \U$ (and $\ell\in \ZZ_{\geq 0}$), the local system $\HH^{\ell}_{\pi}:=R^{\ell}\pi_*\QQ$ and Hodge flag $\F^{\b}$ on $\H^{\ell}_{\pi}:=\HH^{\ell}_{\pi}\otimes \cO_{\U}$ underlie a VHS $\BH^{\ell}_{\pi}:=(\HH^{\ell}_{\pi},\F^{\b}\H^{\ell}_{\pi},\nabla)$ and period map $\mathcal{P}\colon \U_{\CC}^{\text{an}}\to \Gamma\backslash \DD=\Gamma\backslash G(\RR)/M$.  Here $G$ and $\DD$ are the Mumford-Tate group resp.~domain \cite{GGK}, $M\leq G(\RR)$ is compact, and $\Gamma$ is (or contains) the monodromy group attached to $\HH^{\ell}_{\pi}$.  More generally, $\BH\subseteq \BH^{\ell}_{\pi}$ will denote a sub-VHS; we are only interested in variations with are ``motivic'' in this sense.  Write $\uh=(h^{0,\ell},h^{1,\ell-1},\ldots,h^{\ell,0})$ for the Hodge numbers $h^{k,\ell-k}{:=}\mathrm{rk}(\gr_{\F}^k\H)$ of $\BH$.  Denote by $\BQ(0)$ the trivial VHS of weight 0 and rank 1, and by $\BH(p)$ the Tate twist of weight $\ell-2p$, which multiplies coefficients of the local system by $(2\pi\ay)^p$.

Given a mixed Hodge structure (MHS) $V$, let 
\begin{equation}\label{eq2.1}
\Hg(V):=\HMS(\QQ(0),V)\cong V_{\QQ}\cap F^0V_{\CC}
\end{equation}
be the space of Hodge classes.  For $V$ pure of negative weight, we write
\begin{equation}\label{eq2.2}
J(V):=\EMS(\QQ(0),V)\cong {V_{\CC}}/({F_0V_{\CC}+V_{\QQ}})
\end{equation}
for the (rational) Jacobian.  The isomorphism in \eqref{eq2.2} is made explicit by taking an extension
\begin{equation}\label{eq2.3}
0\to V\to E\to \QQ(0)\to 0	
\end{equation}
and lifting $1\in \QQ(0)$ to $e_{\QQ}\in E_{\QQ}$ and $e_F\in F^0E_{\CC}$, so that the difference $e_{\QQ}-e_F$ lies in $V_{\CC}$.  Quotienting out the ambiguities of the lifts gives an element of RHS\eqref{eq2.2}.

An admissible normal function $\nu\in \mathrm{ANF}_{\U_{\CC}^{\text{an}}}(\BH(p))$ is an extension
\begin{equation}\label{eq2.4}
0\to \BH\to \BE_{\nu}\to \BQ(-p)\to 0	
\end{equation}
in the category of admissible variations of MHS on $\U_{\CC}^{\text{an}}$.  Tate-twisting \eqref{eq2.4} by $p$ and restricting to the fiber over $u$ gives an extension of the form \eqref{eq2.3}, hence an element in $J(\BH(p)|_u)$.  This leads to a holomorphic, horizontal section of the Jacobian bundle $\J(\BH(p)):=\H/(F^p\H+\HH(p))$ over $\UAN$, with values denoted $\nu(u)$.  (Later, $\tilde{\nu}$ will mean a multivalued lift of $\nu$ to $\H$.) Addition of sections puts a group structure on $\ANF_{\U_{\CC}^{\text{an}}}(\BH(p))$ agreeing with that defined by the Baer sum.

An algebraic (higher Chow) cycle $\Z\in \CH^p(\X,m)$ homologous to $0$ on fibers $X_u\overset{\imath_u}{\hookrightarrow} \X$ of $\pi$ gives rise to an admissible normal function $\nu_{\Z}\in \ANF_{\U_{\CC}^{\text{an}}}(\BH(p))$ with $\ell=2p-m-1$.  We may view $\Z$ as a relative cycle on $\X\times (\PP^1\setminus\{1\},\{0,\infty\})^m$ of codimension $p$. The value $\nu_{\Z}(u)$ is defined by applying the Abel-Jacobi map\footnote{Here and in \eqref{eq2.7} we implicitly compose with the projection $\BH^{\ell}_{\pi}\twoheadrightarrow\BH$.} (cf.~\cite{KLM})
\begin{equation}\label{eq2.5}
\AJ_{X_s}^{p,m}\colon \CH^p_{\text{hom}}(X_s,m)\to J(\BH(p)|_s)
\end{equation}
to $Z_u:=\imath^*_u\Z$.  The higher Chow groups enjoy isomorphisms $$\CH^p(\X,m)\otimes{\QQ}\cong H^{2p-m}_{\text{Mot}}(\X,\QQ(p))\cong\gr^p_{\gamma}K^{\text{alg}}_m(\X)\otimes{\QQ}$$ with motivic cohomology and $K$-theory (as $\X$ is smooth).  We use ``$K_m$-cycles'' to informally convey the cycle type (with $K_0$ ``classical'' and $K_{>0}$ ``higher'') and weight gap of $m+1$ in \eqref{eq2.4}.

Henceforth we restrict to the case where $\U\subset\PP^1$ (with coordinate $z$) is the complement of a finite set $\Sigma$, $\dim(\X)=\ell+1$, $\BH$ is irreducible and nonconstant of level $\ell$,\footnote{Since $\BH\subseteq\BH^{\ell}_{\pi}$, this simply means that the leading Hodge number $h^{\ell,0}\neq 0$.} and $\X,\U,\pi$ are definable over $\QQ$.  We shall also drop the use of $(\,)^{\text{an}}_{\CC}$ to simplify formulas.  The \emph{Lefschetz correspondence} alluded to in \S\ref{S1} takes the form of an isomorphism
\begin{equation}\label{eq2.6}
\ANF_{\U}(\BH(p))\overset{\cong}{\underset{[\,\cdot\,]}	{\longrightarrow}} \Hg(H^1(\U,\BH(p)))\subseteq \Hg(H^{\ell+1}(\X,\QQ(p)))
\end{equation}
induced by the topological invariant.  The \emph{Beilinson-Hodge Conjecture} (BHC) predicts that the fundamental class map (with $m=2p-\ell-1$)
\begin{equation}\label{eq2.7}
\mathrm{cl}^{p,m}_{\pi}\colon \CH^p(\X,m)\to \Hg(H^1(\U,\BH(p)))
\end{equation}
is surjective.\footnote{To be clear, we consider the cycle groups for the \emph{complex} variety, though by a spread argument all Hodge classes would already be definable over $\QB$.} In general $\mathrm{cl}^{p,m}_{\pi}(\Z)=[\nu_{\Z}]$, and so in our situation all normal functions are expected to come from cycles defined over $\QB$.

Let $\Sigma_0\subseteq\Sigma$ and $\U\overset{\jmath}{\hookrightarrow}\U_0:=\PP^1\setminus\Sigma_0$ be the inclusion.  Write $\IH^1(\U_0,\BH)$ for the Zucker/Saito MHS with underlying $\QQ$-vector space $H^1(\U_0,\jmath_*\HH)$; in particular, $\IH^1(\PP^1,\BH)$ is pure (of weight $\ell+1$) and fits in an exact sequence\vspace{-2mm}
\begin{equation}\label{eq2.8}
0\to \overset{\text{\textcircled{A}}}{\IH^1(\PP^1,\BH)}\to H^1(\U,\BH)\to \overset{\text{\textcircled{B}}}{\oplus_{\sigma\in \Sigma}(H_{\lim,\sigma})_{T_{\sigma}}(-1)}\to 0
\end{equation}
where $H_{\lim,\sigma}$ is the LMHS of $\BH$ at $z=\sigma$ and $(\,)_{T_{\sigma}}$ denotes the coinvariants with respect to the monodromy operator at $\sigma$.  Hodge classes in \textcircled{A} correspond to families of $K_0$-cycles ($r=0$ in \eqref{eq2.7}), and those in \textcircled{B} to families of $K_{>0}$-cycles.  Tate-twisting \eqref{eq2.8} by $(p)$, a given normal function $\nu$ is \emph{nonsingular} at $\sigma$ if $[\nu]$ maps to zero in the corresponding summand of \textcircled{B}.  Normal functions singular \emph{only} at $\Sigma_0$ define a subgroup $\ANF_{\U_0}(\BH(p))\cong \Hg(\IH^1(\U_0,\BH(p)))$.  Note that \textcircled{A} is a direct summand of $H^{\ell+1}(\overline{\X})$ for any smooth completion $\XB\supset\X$.

To relate normal functions to differential equations, suppose that $\HH$ has maximal unipotent monodromy (MUM) at $z=0$, $h^{\ell,0}=1$, and let $\mu\in \Gamma(\PP^1,\F_e^{\ell})$ generate the canonically extended Hodge bundle $\F_e^{\ell}\cong \mathcal{O}_{\PP^1}(h)$ on $\PP^1\setminus\{\infty\}$.  We can normalize $\mu$ to have a $\QQ$-Betti period limiting to $1$ at $z=0$, and so that $(2\pi\ay)^{\ell}\mu$ is defined over $\QQ$.  Let $L\in \QQ\langle D,z\rangle$ denote the Picard-Fuchs operator (of degree $d$) killing $\mu$.  Then to $\nu\in \ANF_{\GG_m}(\BH(p))$ we can associate the \emph{quasi-period}
\begin{equation}\label{eq2.9}
\mathfrak{V}(t):=\langle \tilde{\nu}(z),\mu(z)\rangle,
\end{equation}
which is a well-defined holomorphic function on the universal cover $\widetilde{\UAN}$.  We summarize parts of \cite[Thm.~5.1]{GKS} and \cite[Thm.~5.5]{GK} in
\begin{prop}\label{pr2.10}
The \emph{inhomogeneity} $L\mathfrak{V}=:g$ is a polynomial in $z$ with $\deg(g)\leq d-h$.  Further, if $\nu$ is nonsingular at $0$ then we have $z\mid g$; and if $\nu=\nu_{\Z}$ is motivic (with $\Z$ defined/$K\subset\QB$) then $g\in K[z]$.
\end{prop}

\section{Hypergeometrics}\label{S3}

Next we review a fairly ubiquitous class of variations which are simple enough that one can compute everything:  periods, quasi-periods of normal functions, motivic Gamma functions, etc.  For us they will be prescribed by a pair of $r$-tuples $\fa=(\fa_1,\ldots,\fa_r)$ and $\fb=(\fb_1,\ldots,\fb_r)$ of rational numbers contained in $(0,1]$, with $\fb_r=1$, $\fa_i\neq \fb_j$ ($\forall i,j$), and
\begin{equation}\label{eq3.1}
\textstyle Q_{\infty}(\lambda):=\prod_j(\lambda-e^{2\pi\ay\fa_j})\;\;\text{and}\;\;Q_0(\lambda):=\prod_j(\lambda-e^{-2\pi\ay\fb_j})\;\in\;\QQ[\lambda].
\end{equation}
Our convention is that the weight $\ell=\ell_{\fa,\fb}$ of a hypergeometric variation (HGV) is equal to its level, so that $h^{\ell,0}\neq 0$.  We summarize contributions from \cite{Ka,BH,CG,Fe,RR,Ab} in
\begin{prop}\label{pr3.2}
\textup{(i)} There exists a rank-$r$ polarizable VHS\footnote{While our convention is to work rationally, this may be defined as a $\ZZ$-PVHS.} $\BH_{\fa,\fb}$ over $\PP^1\setminus\{0,1,\infty\}$, together with a section\footnote{A warning is in order:  we are not assuming here that $\HH$ has MUM at $z=0$, though this $\mu$ does have a period \eqref{eq3.3} limiting to $1$ there, and can be defined over $\QQ$ in the situation of part (iv) of the Proposition.} $\mu\in \Gamma(\PP^1,\F_e^{\ell})$, for which the associated Picard-Fuchs operator is \eqref{eq1.1}.  On the punctured unit disk $\Delta^{*}$ it has a period of the form
\begin{equation}\label{eq3.3}
\langle\VE^{\vee}_0,\mu\rangle=f_{\fa,\fb}(z):=\sum_{n\geq 0}\prod_j\frac{[\fa_j]_n}{[\fb_j]_n}z^n,
\end{equation}
where $\VE^{\vee}_0$ is a section of $\HH_{\fa,\fb}^{\vee}|_{\Delta^*}$.

\textup{(ii)} The Hodge numbers may be determined by the zigzag procedure of \cite[\S5]{RR}.  In particular, $\ell=0$ iff $\fa,\fb$ are totally interlaced;\footnote{This means that, identifying $(0,1]$ with $S^1$ (and counting multiplicity), any two $\fb_j$'s are separated by an $\fa_i$ (and vice versa).} while if all $\fb_j=1$ then $\ell=r-1$ and $\uh=(1,1,\ldots,1)$.

\textup{(iii)} The local system $\HH_{\fa,\fb}$ is absolutely irreducible, with $Q_0$ and $Q_{\infty}$ the characteristic polynomials of $T_0$ resp.~$T_{\infty}$, and $\mathrm{rk}(T_1-\text{id})=1$. It is definable over $\ZZ$. The monodromy group $\Gamma_{\fa,\fb}$ is finite iff $\fa,\fb$ are totally interlaced.  Otherwise its $\QQ$-Zariski closure (and thus also the Mumford-Tate group) is $\mathrm{Sp}_r$ or $\mathrm{SO}(h_{\text{even}},h_{\text{odd}})$ for $\ell$ odd resp.~even.

\textup{(iv)} Writing $q_{\infty}/q_{0}=\prod_{i=1}^a(\lambda^{-\gamma_i}-1)/\prod_{i=a+1}^{a+b}(\lambda^{\gamma_i}-1)$ with $\gamma_1,\ldots,\gamma_a\in\ZZ_{<0}$ and $\gamma_{a+1},\ldots,\gamma_{a+b}\in\ZZ_{>0}$, $\BH_{\fa,\fb}$ is motivic (over $\QQ$) if $a=1$ or $a+b$ is odd (equiv.~$\ell$ is even).
\end{prop}

\begin{example}\label{ex3.4}
The Legendre elliptic curve family $\pi\colon\mathcal{E}\to \PP^1\setminus\{0,1,\infty\}$ has $\BH^1_{\pi}\cong \BH_{(\frac{1}{2},\frac{1}{2}),(1,1)}$.
\end{example}

The \emph{Frobenius deformation} \cite{BV,GZ,Ke} associated to a HGV is
\begin{equation}\label{eq3.5}
\Phi(s,z):=\sum_{n\geq 0} \prod_{j=1}^r \frac{[s+\fa_j]_n}{[s+\fb_j]_n}z^{s+n},
\end{equation}
a series\footnote{\emph{a priori} formal, but defining a holomorphic function on $\widetilde{\UAN}\times (\CC\setminus\cup_j\{\ZZ_{\leq 0}-\fb_j\})$.} satisfying the two equations
\begin{equation}\label{eq3.6}
\textstyle L\Phi=z^s\prod_{j}(s+\fb_j-1)\;\;\;\text{and}\;\;\;T_0\Phi=e^{2\pi\ay s}\Phi.
\end{equation}
According to \cite[\S3]{BV}, if we re-index the $\{\fb_j\}_{j=1}^r$ as $\{b_i\}_{i=1}^{\mathfrak{r}}$ with multiplicities $m_i$, then 
\begin{equation}\label{eq3.7}
\textstyle\left\{\frac{\partial^c\Phi}{\partial s^c}(1-b_i,z)\;\right|\left.i=1\ldots,\mathfrak{r}\;\;\text{and}\;\;c=0,\ldots,m_i-1\phantom{\frac{\partial^c\Phi}{\partial s^c}}\mspace{-25mu}\right\}
\end{equation}
give a basis of solutions to $L(\cdot)=0$ at $z=0$.

The following basic result makes HGVs very appealing for the study of normal functions:
\begin{prop}\label{pr3.8}
$\ANF_{\GG_m}(\BH_{\fa,\fb}(p))$ has rank one for exactly one value of $p$, and is otherwise zero.
\end{prop}
\begin{proof}
Since $\mathrm{rk}(T_1-\text{id})=1$, Euler-Poincar\'e \cite[(5.6)]{GK} implies that $\IH^1(\GG_m,\BH_{\fa,\fb})$ also has rank $1$.  Thus it must be a Tate object $\BQ(-p)$ for some $p>0$ depending on the hypergeometric indices $\fa,\fb$.
\end{proof}

\begin{rem}\label{re3.9}
If we relax the condition \eqref{eq3.1} to $q_0,q_{\infty}\in K[\lambda]$ for some abelian extension $K/\QQ$, then $\BH_{\fa,\fb}$ is a $K$-VHS, and its Weil restriction $\mathrm{Res}_{K/\QQ}(\BH_{\fa,\fb})$ yields a ($\QQ$-)VHS of rank $[K{:}\QQ].r$. The finite monodromy condition is a bit more involved, see the beginning of \S\ref{S6}.  We refer to $K=\QQ$ as the \emph{balanced} case and $K\neq \QQ$ as the \emph{unbalanced} case.
\end{rem}

\section{Vignette \#1: Identities}\label{S4}

Consider a motivic VHS $\BH$ over $\U\subset\PP^1$ as in \S\ref{S2}.  We assume that the weight $\ell=2q$ is even, the Hodge numbers are all $1$ (so that the rank $r=2q+1$), and $\HH$ has MUM at $z=0$. Then $H_{\lim,0}$ is Hodge-Tate, with $\mathcal{B}:=(H_{\lim,0})_{T_0}(-1)\cong \QQ(-\ell-1)$, and pulling back \eqref{eq2.8} by the inclusion $\mathcal{B}\hookrightarrow\text{\textcircled{B}}$ yields an extension of $\QQ(-\ell-1)$ by $\IH^1(\PP^1,\BH)$.  Assume further that this extension is split.\footnote{This will occur if any of the following is true: $\BH$ is extremal (i.e.~$\mathrm{ih}^1(\PP^1,\BH)=0$); $\BH$ is hypergeometric (see below); or $\BH$ arises from the middle cohomology of a pencil defined by a tempered Laurent polynomial \cite[\S3]{DK1}.}  Then there exists $\nu\in\ANF_{\PP^1\setminus\{0\}}(\BH(\ell+1))$ with $[\nu]$ generating $\mathcal{B}$, and (choosing $\mu$ as at the end of \S\ref{S2}) an associated quasi-period $\mathfrak{V}=\langle \tilde{\nu},\mu\rangle$.

Next let $z_0\in U(\QQ)$ be a \emph{Hodge point}, which means that there exists a nontrivial Hodge class $\xi\in \Hg(\BH(q)|_{z_0})$. As long as the Picard-Fuchs operator $L$ associated to $\mu$ has no apparent singularity at $z_0$, Griffiths transversality forces $D^j\mu|_{z_0}$ to generate $\gr_{\F}^{2q-j}\H|_{z_0}$ for $j\leq 2q$, and
there is a polynomial $P$ of degree $q$ such that\footnote{As in the introduction, we omit the ``$\nabla$'' when applying differential operators to sections of $\H$.} 
\begin{equation}\label{eq4.1}
(P(D)\mu)(z_0)=-(2\pi\ay)^{-\ell}\xi.
\end{equation}

Now suppose \emph{everything} is motivic: writing $X=X_{z_0}$,
\begin{itemize}[leftmargin=0.8cm]
\item [(i)] $\xi=\mathrm{cl}_{X}^{q}(W)$ for $W\in\CH^q(X)$, defined over $K\subset\QB$; and
\item [(ii)]	$\nu=\nu_{\Z}$ for $\Z\in \CH^{\ell+1}(\overline{\X}\setminus X_0,\ell+1)$ (say, defined over $\QQ$).
\end{itemize}
Then $P$ must have coefficients in $K$, because the $(2\pi\ay)^{\ell}D^j\mu|_{z_0}$ are $\QQ$-de Rham and $\xi$ is represented by $(2\pi\ay)^q\delta_W$, which is $K$-de Rham. (Here $\delta_W$ is the current of integration over $W$.)  Writing $a^W\colon W\to \mathrm{Spec}(K)$ for the structure map, let $\alpha$ denote the image of $Z:=a^W_*(\Z_{z_0}|_W)$ under
\begin{equation}\label{eq4.2}
\AJ^{q+1,2q+1}_{\text{Spec(K)}}\colon \CH^{q+1}(K,2q+1)\to \CC/\QQ(q+1).
\end{equation}
Pairing \eqref{eq4.1} with a lift $\tilde{\nu}_{\Z}$  gives a number
\begin{equation}\label{eq4.3}
\tilde{\alpha}:=-(P(D)\mathfrak{V})(z_0)=\left<\tilde{\nu}_{\Z}(z_0),\frac{\delta_W}{(2\pi\ay)^q}\right>\in \CC
\end{equation}
projecting to $\alpha$. In the event that $K=\QQ$, then by Borel's theorem $\tilde{\alpha}$ must be a rational multiple of $\zeta(q+1)$.\footnote{\emph{A priori}, for $q+1$ odd, it could also be in $\QQ(q+1)$.  However, in practice (when $K=\QQ$) LHS\eqref{eq4.3} will be a series with $\QQ$-coefficients, and this is not possible.}

\begin{rem}
There are numerous examples where (ii) is known: for instance, if $\X$ arises from compactifying level sets of a tempered Laurent polynomial on a torus $\GG_m^{\ell+1}$, then $\Z$ is essentially given by the coordinate symbol $\{x_1,\ldots,x_{\ell+1}\}$, cf.~\cite[\S6]{GK}.  (The hypergeometric examples below typically arise in this fashion, and this is often easy to prove using Hadamard products as in [op.~cit.].)  In that context, if (i) is unknown, \emph{an identity of the form $(P(D)\mathfrak{V})(z_0)\in \zeta(q+1).\QQ$ provides evidence toward it} (with $K=\QQ$), \emph{and thus toward the Hodge Conjecture for $X_{z_0}$}.
\end{rem}

Specializing to the hypergeometric setting, we take $\BH=\BH_{\fa,\fb}$ with $r=\ell+1$ and all $\fb_j=1$, and $L,\mu$ as in \S\ref{S3}.  Since the weight is even, $\BH$ is motivic by Prop.~\ref{pr3.2}(iv).  We have 
\begin{equation}
\IH^1(\GG_m,\BH)\cong (H_{\lim,0})_{T_0}(-1)\cong\QQ(-\ell-1)
\end{equation}
and $p=\ell+1$ in Prop.~\ref{pr3.8}.  Let $\nu$ be a generator of the rank one group $\ANF_{\PP^1\setminus\{0\}}(\BH(\ell+1))$, which we assume is motivic (see the last Remark).  Since $\deg(L)=1=\deg(\F^{\ell}_e)$, by Prop.~\ref{pr2.10}
\begin{equation}\label{eq4.6}
L\mathfrak{V}=1	
\end{equation}
after Galois descent and rescaling.

Now since $\fa\subset(0,1)$, there exists a unique single-valued determination of $\mathfrak{V}$ about $z=\infty$.  It exists because $\nu$ is nonsingular at $\infty$, cf.~\cite[\S4.2]{GKS}. It is unique because ambiguities of $\mathfrak{V}$ lie in the periods of $\mu$, none of which are single-valued there.  Writing $w:=\frac{1}{z}$,
\begin{equation}\label{eq4.7}
\textstyle \hat{L}:=L\frac{1}{z}=\prod_{j=1}^r (D_w+\fa_j)-w(D_w+1)^r
\end{equation}
is the PF operator associated to $z\mu$, and sends $z\mathfrak{V}\mapsto 1$.  Notice that $(\fa,\fb)=(\fa,\bm{1})$ have been replaced by $(\hat{\fa},\hat{\fb})=(\bm{1},\fa+1)$, so the corresponding Frobenius deformation is 
\begin{equation}
\hat{\Phi}(s,w):=\sum_{n\geq 0}\frac{[s+1]_n^r}{\prod_j[s+\fa_j+1]_n}w^{s+n}.
\end{equation}
By \eqref{eq3.6} (or direct calculation), we have
\begin{equation}
\hat{L}\hat{\Phi}(s,w)=w^s\prod_{j=1}^r(s+\fa_j),
\end{equation}
and setting $s=0$ gives that
\begin{equation}\label{eq4.10}
\frac{1}{\prod_j\fa_j}\hat{\Phi}(0,w)=\frac{1}{w}\sum_{n\geq 1} \frac{(n!)^rw^n}{n^r \prod_j [\fa_j]_n}
\end{equation}
is a single-valued solution to $\hat{L}(\cdot)=1$ about $z=\infty$.  But $\hat{L}(\cdot)=0$ has no nontrivial single-valued solutions there, so $\hat{L}(z\mathfrak{V}-\eqref{eq4.10})=0$ implies that $z\mathfrak{V}=\eqref{eq4.10}$.  Putting this together with the previous argument yields

\begin{thm}\label{th4.11}
\textup{(i)} On $|z|>1$, the unique single-valued lift $\tilde{\nu}$ yields 
\begin{equation}\label{eq4.12}
\mathfrak{V}(z):=\langle\tilde{\nu}(z),\mu(z)\rangle=\sum_{n\geq 1}\frac{(n!)^rz^{-n}}{n^r\prod_j[\fa_j]_n}.
\end{equation}

\textup{(ii)} Let $z_0\in \QQ$ be a Hodge point for $\BH_{\fa,\fb}$ with $|z_0|>1$, at which the Hodge conjecture holds and the cycle $W$ is defined over $\QQ$. Then there is a polynomial $P\in \QQ[T]$ of degree $q$ and $\kmath\in\QQ^*$ such that
\begin{equation}\label{eq4.13}
\sum_{n\geq 1}\frac{P(n)(n!)^rz_0^{-n}}{n^r\prod_{j=1}^r[\fa_j]_n}=\kmath\cdot\zeta(q+1).
\end{equation}
\end{thm}

\begin{example}\label{ex4.14}
Elmi and Golyshev \cite{EG} have carried out a computer search for atypical Hodge points in HGVs.  Among the very few examples with $\fb=\bm{1}$ are
\begin{itemize}[leftmargin=0.5cm]
\item $\fa=(\tfrac{1}{2}, \tfrac{1}{2}, \tfrac{1}{2}, \tfrac{1}{2}, \tfrac{1}{2})$ and $z_0=-4$
\item $\fa=(\tfrac{1}{5},\tfrac{2}{5}, \tfrac{1}{2}, \tfrac{1}{2}, \tfrac{1}{2}, \tfrac{1}{2}, \tfrac{1}{2},\tfrac{3}{5},\tfrac{4}{5})$ and $z_0=-\frac{3125}{1024}$
\end{itemize}
which correspond exactly to \eqref{eq1.2} and \eqref{eq1.3}, proved in \cite{Gu} and \cite{Au} via analytic techniques.
\end{example}

\begin{rem}\label{re4.15}
There is a related perspective on a conjecture of Gross, Kohnen, and Zagier \cite{GKZ}.  Over a modular curve $Y=Y(\Gamma)$ we have a family of Picard-rank 19 $K3$ surfaces with $\BH=\BH^2_{\text{tr}}$ ($\uh=(1,1,1)$). If $S_4(\Gamma)=\{0\}$ then every CM point $\tau\in Y(\QB)$ yields a normal function $\nu_{\tau}$ of $K_1$ type. Pairing with a section of $\H_{\RR}^{1,1}$ gives $\mathfrak{G}_{\tau}$, a higher Green's function. 

Pick any other CM point $\tau'$.  By Lefschetz (1,1), the extra Hodge class on $X_{\tau'}$ is algebraic. If BHC holds, we have $\mathfrak{G}_{\tau}(\tau')\in \QB\langle\log\QB\rangle$, which is (a coarsening of) the GKZ conjecture.  That this is now a theorem of Bruinier-Li-Yang \cite{BLY} this ``gives good evidence for BHC.''  See \cite{DK2} for references and details of this argument, and direct proofs of the relevant cases of the BHC for several families of $K3$ surfaces.
\end{rem}

\section{Lefschetz correspondence revisited}\label{S5}

In this brief intermezzo we want to amplify the material in \S\ref{S2} in two ways. First, we reinterpret the relationship between Hodge classes and normal functions using middle convolution with a mixed Hodge module incarnation of Katz's ``Ur-object'' \cite{Ka}. Second, we explain how to layer integrality onto normal functions and their invariants when the VHS has an underlying $\ZZ$-local system.

We continue to assume that $\HH$ is irreducible and nonconstant; in particular, $\HH$ has trivial fixed part ($\Gamma(\U,\HH)=\{0\}$). Fix a smooth completion $\XB\supset\X$ on which $\pi$ extends to $\pb\colon \XB\to \PP^1$, and set $\XB_{\U_0}:=\pb^{-1}(\U_0)$ for any $\U_0\supseteq \U$. Writing $\kappa\colon \U\hookrightarrow\GG_m$, $\kappa'\colon\GG_m\setminus\{1\}\hookrightarrow\GG_m$, and $\imath\colon \{1\}\hookrightarrow\GG_m$ for the inclusions, the MHM $\kappa_{!*}\BH$ has underlying (shifted perverse) sheaf $\kappa_*\HH$, and $\UR:=R\kappa'_!\BQ$ sits in an exact sequence $0\to R\imath_*\BQ\to\UR\to \BQ[1]\to 0$. Multiplicative convolution on $\GG_m$ then yields the sequence in $\mathrm{AVMHS}(\U)$
\begin{equation}\label{eq5.1}
0\to \BH\to \BH*\UR	\to \mathbf{IH}^1(\GG_m,\BH)\to 0
\end{equation}
in which the right-hand term is constant.  The fiber of the middle term at $z$ may be understood as a subquotient\footnote{In case $\HH=R^{\ell}\pi_*\QQ_{\X}$, it is just the quotient by the \emph{phantom cohomology} $\oplus_{\sigma\in \Sigma\cap\GG_m}\ker\{H^{\ell+1}(X_{\sigma})\to H_{\lim,\sigma}^{\ell+1}\}$. We can also write in general $(\BH*\UR)_z=\IH^1(\GG_m,\{z\};\BH)$.} MHS of $H^{\ell+1}(\XB_{\GG_m},X_z)$ as in \cite[Prop.~4.15]{GKS}.

Now we can think of any Hodge class $\alpha\in \Hg(\IH^1(\GG_m,\BH(p)))$ as an inclusion $$\alpha\colon \QQ(-p)\hookrightarrow \IH^1(\GG_m,\BH).$$ Pulling back \eqref{eq5.1} gives a normal function $\nu$ as in \eqref{eq2.4}, with $[\nu]=\alpha$. Hence \emph{Ur-convolution} provides an inverse to \eqref{eq2.6} which makes precise the notion of ``fibering out a Hodge class by a normal function''. Note that the rank of $\IH^1(\GG_m,\BH)$ is given by the sum of the ranks of the monodromy coinvariants over $\Sigma\cap \GG_m$, which for a hypergeometric VHS is always $1$ as in Proposition \ref{pr3.8}.

In the event that $\HH$ has an underlying $\ZZ$-local system $\HH_{\ZZ}$, we write $\ANF_{\U}(\BH(p))_{\ZZ}\subset \ANF_{\U}(\BH(p))$ for normal functions arising from integral Hodge classes. Consider the complexes of sheaves
\begin{multline*}
\SC=\{\H\overset{\nabla}{\to}\H\otimes\Omega^1_{\U}\},\;\;\;\SF=\{\F^p\H\overset{\nabla}{\to}\F^{p-1}\H\otimes\Omega^1_{\U}\} \\ \text{and}\;\; \SJ=\mathrm{Cone}\{\HH_{\ZZ}(p)\oplus\SF\to \SC\}.
\end{multline*}
The holomorphic quasi-horizontal sections of the integral Jacobian bundle $\H/(F^p\H+\HH_{\ZZ}(p))$ are then given by $\NF_{\U}(\BH(p))_{\ZZ}:=\HH^0(\SJ)$, with the connecting homomorphism for $\HH^0$ producing 
\begin{equation}\label{eq5.2}
\xymatrix{\NF_{\U}(\BH(p))_{\ZZ} \ar [rd]_{0} \ar [r]^{(\delta,[\cdot])\mspace{70mu}} & \HH^1(\SF)\oplus H^1(\U,\HH_{\ZZ}(p)) \ar [d] \\ & \HH^1(\SC)\cong H^1(\U,\HH_{\CC}),}
\end{equation}
where the \emph{infinitesimal invariant}\footnote{Typically one defines $\delta\nu$ by using the edge homomorphism to pass to $\Gamma(\U,\H^1_{\nabla}(\SF))$. For our application this distinction is unnecessary.} $\delta\nu$ is computed by $\nabla\tilde{\nu}$. That the diagonal is zero is due to composing two maps in a long-exact sequence; it expresses the fact that $[\nu]$ is a Hodge class.

Admissible normal functions admit refined infinitesimal and topological invariants. Write ${\jmath}\colon \U\hookrightarrow\U_0$, $\bar{\jmath}\colon \U\hookrightarrow\PP^1$, and $\bar{\jmath}^0\colon \U_0\hookrightarrow\PP^1$; denote by $t_{\sigma}$ a local coordinaate at $\sigma\in\Sigma$ and put $D_{\sigma}=\frac{t_{\sigma}}{2\pi\ay}\frac{d}{dt_{\sigma}}$. The upper\footnote{The lower canonical extension $(\H_{e},\nabla_e)$, to be briefly used in \S\ref{S6}, replaces $[0,1)$ by $(-1,0]$.} canonical extension $(\H^e,\nabla^e)$ of $(\H,\nabla)$ to $\PP^1$ is obtained by requiring eigenvalues of $\nabla_{D_{\sigma}}$ to lie in $[0,1)$ for each $\sigma$. A similar argument with logarithmic complexes (denoting log poles by $\langle\,\cdot\,\rangle$)
\begin{multline*}
\SC^e:=\{\H^e\overset{\nabla^e}{\to}\H^e\otimes \Omega^1_{\PP^1}\langle\Sigma\rangle\},\;\;\SF^e:=\{\F^p\H^e\to \F^{p-1}\H^e\otimes \Omega^1_{\PP^1}\langle\Sigma\rangle\},\\ \text{and}\;\;\SJ^e:=\text{Cone}\{R\bar{\jmath}_*\HH_{\ZZ}(p)\oplus\SF^e\to \SC^e\}
\end{multline*}
yields the refined $(\delta,[\cdot])$. The main point is that, by admissibility, $\nabla\tilde{\nu}$ factors through $\H^1_{\nabla}(\SF^e)$ (see for example \cite[\S2.11]{KP}).

For normal functions nonsingular on $\Sigma\cap \U_0$, this is further refined by replacing $R\bar{\jmath}_*\HH_{\ZZ}(p)$ by $(R\bar{\jmath}^0_*)\jmath_*\HH_{\ZZ}(p)$ and $\SC^e$ by the subcomplex $\SC^e_{\U_0}$ of sections which are locally $L^2$ on $\Sigma\cap \U_0$ (and similarly for $\SF^e$).  According to \cite{Zu}, the $L^2$ requirement affects only the unipotent part of $\H^e$ locally, on which eigenvalues of $\nabla_{D_{\sigma}}$ lie in $\ZZ_{\geq 0}$. Writing $\V_{\sigma}^e$ for this and $\mathcal{W}_{\b}$ for the weight monodromy filtration centered about $\ell$, we use (for $\SC^e_{\U_0}$) the subcomplex $$\{\mathcal{W}_{\ell}\V_{\sigma}^e+t_{\sigma}\V^e_{\sigma}\}\to \{\mathcal{W}_{\ell-2}\V_{\sigma}^e+t_{\sigma}\V_{\sigma}^e\}\otimes\Omega^1\langle\sigma\rangle.$$ This produces a diagram
\begin{equation}\label{eq5.3}
\xymatrix{\ANF_{\U_0}(\BH(p))_{\ZZ} \ar [rd]_{0} \ar [r]^{(\delta,[\cdot])\mspace{80mu}} & \HH^1(\PP^1,\SF^e_{\U_0})\oplus \IH^1(\U_0,\HH_{\ZZ}(p)) \ar [d] \\ & \IH^1(\U_0,\HH_{\CC}).}
\end{equation}
For $\sigma\in \PP^1\setminus \U_0$, the residue map 
\begin{equation}\label{eq5.4}
\Res_{\sigma}\colon \Hg(\IH^1(\U_0,\HH_{\ZZ}(p)))\to \Hg((H_{\lim,\sigma})_{T_{\sigma}}(p-1))
\end{equation}
computes the singularity class
\begin{equation}\label{eq5.5}
\text{sing}_{\sigma}(\nu)=\Res_{\sigma}([\nu])\underset{\otimes\QQ}{=}\Res_{\sigma}(\delta\nu).
\end{equation}
We note that \eqref{eq5.4} is induced by the more geometric residue map $H^{\ell+1}(\XB_{\U_0})\to H_{\ell}(X_{\sigma})(-\ell-1)$.

\section{Vignette \#2: Algebraicity}\label{S6}

Let $\fa,\fb\subset (0,1]\cap \QQ$ with $\fa\cap\fb=\emptyset$, $\fb_r=1$, and write $m$ for the lcm of the denominators. Given any $k\in\ZZ\cap[0,m]$ coprime to $m$, we consider the pairs $\langle k\fa\rangle,\langle k\fb\rangle$ where $\langle\cdot\rangle$ takes the part in $(0,1]$. Writing $K\subseteq\QQ(\zeta_m)$ for the field generated by the coefficients of $q_0,q_{\infty}$, this produces $s=[K{:}\QQ]$ distinct hypergeometric conjugates $f_1=f,f_2,\ldots,f_s$ of $f:=f_{\fa,\fb}$. If $\langle k\fa\rangle$ and $\langle k\fb\rangle$ interlace on the unit circle for each $k$ as above, then the monodromy group of the (irreducible) hypergeometric local system $\VV=\VV_{\fa,\fb}$ associated to $f$ is finite and defined over $\mathcal{O}_K$ \cite[Thms.~3.5 \& 4.8]{BH}. It underlies an isotrivial $\mathcal{O}_K$-VHS $\BV$ of weight and level zero.

In particular, $f$ itself has finite monodromy. By the Riemann existence theorem, it is algebraic over $\CC(z)$; and since it has rational power-series coefficients at $0$, it is algebraic over $\QQ(z)$. That is, we have an absolutely irreducible polynomial $P\in \QQ[z,y]$ defining an affine curve, whose normalized completion is a minimal \emph{existence domain} for $f$. This is a (smooth, irreducible) curve cover $\XB\overset{\pb}{\to}\PP^1_z$, defined over $\QQ$ and branched over $\{0,1,\infty\}$, together with a function $\tilde{f}\in\QQ(\XB)$ (namely, $y$), which on $\XB^{\text{an}}_{\CC}$ is a well-defined meromorphic continuation of $f$. Since the monodromies of the $f_j$ are Galois conjugate, the existence domains are the same and we have $\tilde{f}_j\in \QQ(\XB)=\QQ(\tilde{f},z)$ as well.

The hypergeometric function $f$ is \emph{factorial} if it can be written in the form $\sum_{n\geq 0}\frac{(a_1n)!\cdots(a_un)!}{(b_1n)!\cdots(b_vn)!}(cz)^n$. It is known that this is equivalent to the balanced case where \eqref{eq3.1} holds and $K=\QQ$ \cite{DR1}. In this case, the following result was conjectured by Golyshev (see \cite{Za}), and its first statement proved by Delaygue and Rivoal \cite[Thm.~2]{DR1} and Bostan \cite[Cor. 2.3]{Bo} using $p$-adic analysis.  We offer instead a Hodge-theoretic proof using normal functions, which gives the full result:\footnote{In \cite{DR1} it is only shown that $F^m\in \QQ(f,z)$ for some $m\in \mathbb{N}$.  Note that some choice of constant of integration and extension of $F$ over $\XB$ is required for the statement to make sense, and the proof implicitly provides these.}

\begin{thm}\label{th6.1}
Let $\fa,\fb$ be balanced and totally interlaced, and $f=f_{\fa,\fb}(z)$. Then $F:=\exp(\int f\frac{dz}{z})$ is algebraic over $\QQ(z)$; more precisely, we have $F\in \QQ(f,z)$.
\end{thm}

\begin{proof}
Let $\tilde{R}^0\pi_*\ZZ_{\X}$ be the local system on $\U=\PP^1\setminus\{0,1,\infty\}$ representing the reduced (trace-free) $H^0$ of fibers. The hypergeometric local system $\HH_{\ZZ}\subseteq \tilde{R}^0\pi_*\ZZ_{\X}$ underlies a weight $0$ isotrivial $\ZZ$-VHS $\BH$.  Up to torsion,\footnote{There is a tricky point here that the generator of the free part of $\IH^1(\PP^1\setminus\{0\},\HH_{\ZZ}(1))$ maps in general to a nontrivial integer multiple of the generator of the free part of $\IH^1(\GG_m,\HH_{\ZZ}(1))$.  See Step 3 below.} we have 
$$\IH^1(\PP^1{\setminus}\{0\},\BH(1))\cong\IH^1(\GG_m,\BH(1))\cong \ZZ(0);$$ 
in particular, the first two Hodge structures have rank one.  This is really \emph{the} key observation: while $\XB$ may have genus $>0$ (see for example \cite{Ko}), and $H^1(\X)$ may contain nontrivial extensions, there is always this Hodge class.  We remark that the rank $r$ of $\HH$ is always less, and typically \emph{much} less,\footnote{According to \cite{RV}, the example in Q2 from \S\ref{S1} (with $r=8$) has $d=483840$.} than the mapping degree of $\pi$, which we will denote by $d$.  We show how to go from this Hodge class to a normal function in three steps.

\textbf{Step 1: de Rham.} Write $\U=\PP^1\setminus\{0,1,\infty\}\overset{\jmath}{\hookrightarrow}\U_0=\PP^1\setminus\{0\}\overset{\bar{\jmath}^0}{\hookrightarrow}\PP^1$. Since the monodromies of $\HH$ are finite at each puncture, the de Rham complex quasi-isomorphic to $(R\bar{\jmath}^0_*)\jmath_*\HH_{\CC}$ is
\begin{equation}
\SC_{\U_0}^e:=\{\H^e\to \H_e'\otimes\Omega^1_{\PP^1}\langle 0\rangle\}.
\end{equation}
Here $\H^e$ is the usual upper canonical extension to $\PP^1$, and $\H_e'$ is the lower canonical extension at $\{1,\infty\}$ and upper canonical extension at $\{0\}$. (This is nothing more than the specialization of the complex at the end of \S\ref{S5} to our present situation.) So we have $\HH^1(\PP^1,\SC^e_{\U_0})\cong \IH^1(\U_0,\HH_{\CC})$, and $\HH^1$ of the subcomplex $\SF^e_{\U_0}:=(\H_e'\otimes \Omega^1_{\PP^1}\langle 0\rangle)[-1]$ computes $F^1\IH^1(\U_0,\HH_{\CC})=\IH^1(\U_0,\HH_{\CC})\cong \CC$. Moreover, a straightforward analysis\footnote{$\XB$ must have an order-$k$ ramification in order for $\tilde{f}$ to have local exponent $-\frac{a}{k}$ (with $a\in (0,k)\cap\ZZ$), and $t_{\sigma}^{-\frac{a}{k}}dt_{\sigma}$ lifts to a holomorphic form at such a point.}  gives that $\Gamma(\PP^1,\H_e'\otimes\Omega^1_{\PP^1}\langle 0\rangle)\hookrightarrow \Omega^1(\XB)\langle X_0\rangle$, where $X_0=\pb^{-1}(0)$ is a finite point-set.

According to \cite{BH}, the indices of the hypergeometric $\mathcal{D}$-module at $0,1,\infty$ are $\{1-\fb_i\}_{i=1}^r$, $\{0,1,2,\ldots,r-1,-\tfrac{1}{2}\}$, and $\{\fa_i\}_{i=1}^r$ respectively.  This gives the local exponents of all analytic continuations of $f$. Moreover, one may interpret $\tilde{f}|_{\X}$ as a (well-defined) section of $\H$ over $\U$, with ``monodromies of the function'' arising from pairing this section with (multivalued) sections of $\HH_{\ZZ}^{\vee}$. Thus we see that over $\PP^1$, $\tilde{f}$ gives a section of the extension $\H_e''$ given by $\H^e$ at $0$, $\H_e$ at $1$, and $\H^e$ at $\infty$ (but with no monodromy eigenvalue-$1$ part at $\infty$). Conclude that $\tilde{f}\frac{dz}{z}$ belongs to $\Gamma(\PP^1,\H_e'\otimes \Omega^1_{\PP^1}\langle0\rangle)$, and is thus an element of $\Omega^1(\XB)\langle X_0\rangle$ giving a multiple of the Hodge class generating $\IH^1(\U_0,\BH(1))$.  At this stage, for all we know, that multiple could be complex.  For the theorem of \cite{DR1}, we need to show it is rational; for the full result, one might try to show it is integral, but this is false.  Instead we will show that it is an \emph{integral generator} of $\IH^1(\GG_m,\BH(1))$.\vspace{2mm}

\textbf{Step 2: Integrality.} As only one solution of the hypergeometric $\mathcal{D}$-module is single-valued about $0$, that solution must be part of any basis with respect to which the monodromy group is actually represented by integral matrices (in the sense of \cite[Thm.~3.5]{BH}). So an arbitrary monodromy beginning from and returning to a neighborhood of $z=0$ carries $f=:g_1$ to $\sum_{i=1}^r a_ig_i$, where $g_i$ has local exponent $1-\fb_i$ and $a_i\in\ZZ$ ($\forall i$). Since $g_i(0)=0$ for $i>1$, we deduce that the restriction of $\tilde{f}$ to $X_0$ defines a divisor $\VE_0$ on $\XB$, with $\ZZ$-coefficients, and defined over $\QQ$. As $f(0)=1$, one of these coefficients is $1$, and the period of $\tilde{f}\frac{dz}{z}$ over a loop $\ell_0\subset\XB\setminus X_0$ encircling this point of $X_0$ is $2\pi\ay$. Thus $\tilde{f}\frac{dz}{z}$ is at least in $\IH^1(\U_0,\HH(1))=\IH^1(\GG_m,\HH(1))$.

Write $\U\overset{\kappa}{\hookrightarrow}\GG_m\overset{\bar{\kappa}^0}{\hookrightarrow}\U_{\infty}:=\PP^1\setminus\{\infty\}$ and $\jmath':=\bar{\kappa}^0\circ\kappa$ for the inclusions. In the next paragraph we show that $\tilde{f}\frac{dz}{z}$ lifts to $\IH^1(\GG_m,\HH_{\ZZ}(1))=H^1(\GG_m,\kappa_*\HH_{\ZZ}(1))$. Because of its residue of $\VE_0$ at $0$, it will then generate the free part of this group. The lift also implies that $\tilde{f}\frac{dz}{z}|_{\XB\setminus (X_0\cup X_{\infty})}$ has all its periods in $\ZZ(1)$. But since $\tilde{f}\frac{dz}{z}$ is smooth along $X_{\infty}$, with trivial residues in $H^0(X_{\infty},\ZZ)$, by localization $\tilde{f}\frac{dz}{z}|_{\XB\setminus X_0}$ has all its periods in $\ZZ(1)$, thereby giving a class in $H^1(\XB\setminus X_{\infty},\ZZ(1))$.

Consider the Leray spectral sequence for $\bar{\kappa}^0_*$ applied to the sheaf $\kappa_*\HH_{\ZZ}(1)$.  Because $H^2(\U_{\infty},\jmath'_*\HH_{\ZZ}(1))=\{0\}$, this becomes an exact sequence of abelian groups
\begin{equation}\label{eqn1}
0\to\IH^1(\U_{\infty},\HH_{\ZZ}(1))\to \IH^1(\GG_m,\HH_{\ZZ}(1))\overset{\Res_0}{\to}(H_{\ZZ})_{T_0}\to 0,
\end{equation}
where $(H_{\ZZ})_{T_0}=\ZZ\langle\VE_0\rangle\oplus\text{torsion}$ (as $\fb_r=1\neq$ other $\fb_j$'s). So there exists a class $\alpha\in \IH^1(\GG_m,\HH_{\ZZ}(1))$ with $\Res_0(\alpha)=\VE_0$. In $\eqref{eqn1}\otimes \QQ$, $\IH^1(\U_{infty},\HH(1))=\{0\}$ by Euler-Poincar\'e and $\tilde{f}\frac{dz}{z}\overset{\Res_0}{\longmapsto}\VE_0$.  It follows at once that $\alpha$ maps to $\tilde{f}\frac{dz}{z}$ in the commutative diagram $\eqref{eqn1}\to\eqref{eqn1}\otimes \QQ$, as claimed.  We remark that $\tilde{f}\frac{dz}{z}$ does not lift further \emph{integrally} to $\U_0$, since in 
\begin{equation}\label{eqn2}
0\to \IH^1(\U_0,\HH_{\ZZ}(1))\to \IH^1(\GG_m,\HH_{\ZZ}(1))\to (H_{\ZZ})_{T_{\infty}}\to 0
\end{equation}
the last term (and the image of $\VE_0$ hence $\alpha$ in it) is torsion but not trivial.\vspace{2mm}

\textbf{Step 3: Normal function.} From the result of Step 2 that $\tilde{f}\frac{dz}{z}$ represents a class in $H^1(\XB\setminus X_0,\ZZ(1))$, it is already clear that $\exp(\int\tilde{f}\frac{dz}{z})$ yields a well-defined function $\tilde{F}\in\CC(\XB)$ with divisor $(\tilde{F})=\VE_0$. But we get a bit more information about $\tilde{F}$ if we construct it as a normal function. Indeed, the Ur-convolution \eqref{eq5.1} yields a \emph{tautological} generator
\begin{equation}
\nu\in\ANF_{\GG_m}(\BH(1))_{\ZZ}/\text{tors.}\cong \ZZ.
\end{equation}
We can think of $\nu$ as a section of the left-hand side of
$$\H/\HH_{\ZZ}(1)\overset{\cong}{\underset{\exp}{\longrightarrow}}\HH_{\ZZ}\otimes \mathcal{O}^*,$$
and define $\tilde{F}:=\exp(\nu)$.
Now $[\nu]$ and $[\tilde{f}\tfrac{dz}{z}]$ represent the same class in $\IH^1(\GG_m,\HH_{\ZZ}(1))$, so by \eqref{eq5.5} $\text{sing}_{\sigma}([\nu])=\Res_{\sigma}([\tilde{f}\tfrac{dz}{z}])=\VE_0$ for $\sigma=0$ and is zero for $\sigma=1$. The situation at $\infty$ is more subtle: the failure of $\tilde{f}\frac{dz}{z}$ to lift in \eqref{eqn2} indicates a torsion class $\text{sing}_{\infty}([\nu])$, which corresponds to $\tilde{F}$ taking root-of-1 values in $X_{\infty}$ (rather than having $\tilde{F}|_{X_{\infty}}\equiv 1$).  In any case, via $\eqref{eq5.3}\otimes \QQ$, $\delta\nu=\nabla\tilde{\nu}=\text{dlog}(\tilde{F})$ factors through $\HH^1(\SF^e_{\U_0})$; and so $\text{dlog}(\tilde{F})$ is an element of $\Omega^1(\XB)\langle X_0\rangle$ representing the same class as $\tilde{f}\frac{dz}{z}$. Thus they are equal as 1-forms.

The key bonus from arguing this way is that since $\tilde{F}$ is a section of $\HH_{\ZZ}\otimes \mathcal{O}^*\subseteq\tilde{R}^0\pi_*\ZZ_{\X}\otimes\mathcal{O}_{\U}^*$, the multiplicative trace must satisfy $\pi_*^{\times}\tilde{F}\equiv 1$. Also, \eqref{eq5.5} yields again $(\tilde{F})=\Res_{X_0}(\text{dlog}(\tilde{F}))=\VE_0$. These two features of $\tilde{F}$ are $\mathrm{Gal}(\CC/\QQ)$-invariant, so by spreading out and specializing we may assume $\tilde{F}\in L(\XB)$ for a number field $L\supseteq\QQ(\zeta_d)$. For any $\sigma\in \mathrm{Gal}(L/Q)$, we have $({}^{\sigma}\tilde{F})={}^{\sigma}(\tilde{F})={}^{\sigma}\VE_0=\VE_0$ (since $\VE_0$ is defined over $\QQ$ with $\ZZ$-coefficients), and so ${}^{\sigma}\tilde{F}/\tilde{F}=:c_{\sigma}\in L$ is a constant. Since $\pi^{\times}_*\,{}^{\sigma}\tilde{F}={}^\sigma\pi_*^{\times}\tilde{F}\equiv 1$, $c_{\sigma}^{d}=1$ and $c_{\sigma}\in L^{\times}$; so $\{c_{\sigma}\}$ defines a $1$-cocycle in the trivial group $H^1(\mathrm{Gal}(L/\QQ),L^{\times})$. Thus it comes from some $\gamma\in L$ with ${}^{\sigma}\gamma/\gamma=c_{\sigma}$ ($\forall\sigma$), and replacing $\tilde{F}$ by $\tilde{F}/\gamma$ makes it $\mathrm{Gal}(L/\QQ)$-invariant hence an element of $\QQ(\XB)=\QQ(\tilde{f},z)$.
\end{proof}

\begin{rem}
One can also argue that the correct determination of $\tilde{F}$ is the one given by integrating from the ramification point(s) of $\pb$ over $z=1$. As an invariant section, the component of $\lim_{z=1}\nu$ in the rank-one nonunipotent part of $\H^e|_{z=1}$ must be 2-torsion. Thus $\tilde{F}=\pm 1$ at the ramification points, which are collectively defined over $\QQ$. So we can pick one and normalize $F$ to $+1$ there, and in that sense we have $F=\exp(\int_1\tilde{f}\frac{dz}{z})$ on a neighborhood of $0$. This is also consistent with $\text{sing}_1([\nu])=0$ in Step 3 above.
\end{rem}

\begin{example}
The simplest example is already interesting:  let $r=1$, $\fa=\frac{1}{2}$, and $\fb=1$.  The hypergeometric ``period function'' is
\begin{equation}
f(z)=\sum_{n\geq 0}\frac{[\frac{1}{2}]_n}{n!}z^n=\frac{1}{\sqrt{1-z}}
\end{equation}
and $\XB$ is defined by $P(z,y)=y^2-zy^2-1$.  The exponential integral is
\begin{equation}
F(z)=\exp\left(\int_1\frac{dz}{z\sqrt{1-z}}\right)=\frac{1-\sqrt{1-z}}{1+\sqrt{1-z}}
\end{equation}
and thus $\tilde{F}=\frac{y-1}{y+1}$. The normal function identifies with
\begin{equation}
\log(F)=\log\left(\frac{z}{4}\right)	+\sum_{n>0}\binom{2n}{n}\frac{z^n}{4^n n}.
\end{equation}
Its specializations to the rational points $z=\frac{4}{N}$ with $N>5$ and $N(N-4)$ squarefree have meaning in the context of Dirichlet's class number formula for $K_N:=\QQ(\sqrt{N(N-4)})$. Indeed, $F(\frac{4}{N})$ is a fundamental unit and $-\log F(\frac{4}{N})$ the Dirichlet regulator, so that 
\begin{equation}
\log F(\tfrac{4}{N})=2\frac{\zeta_{K_N}'(0)}{h_{K_N}}.
\end{equation}
\end{example}

\begin{rem}
One might ask what arithmetic quantities are ``interpolated'' by $\int\frac{dz}{z}$ of more general balanced hypergeometric periods. Already for something as simple as $$r(t)=\log(t)+\sum_{n>0}\binom{2n}{n}^2\frac{t^n}{n}$$ this gets into Bloch-Kato for $K_2$ of (families of) elliptic curves \cite{dJDGK}; while one can test Bloch-Beilinson in families of Calabi-Yau 3-folds using more complicated determinants of hypergeometric expressions \cite{GK}. I don't know if Lefschetz could have foreseen normal functions being used to experimentally verify conjectures in arithmetic geometry, but they are proving to be an indispensable tool.
\end{rem}

Turning to the unbalanced case, we have the following result, also proved by Delaygue--Rivoal \cite[Thm.~2]{DR2}.

\begin{thm}\label{th6.11}
Let $\{f_j\}_{j=1}^s$ be the complete set of hypergeometric series conjugate to $f=f_1$. Then for any $\alpha\in\mathcal{O}_K$, $$F_{\alpha}:=\exp\left(\sum_{j}\sigma_j(\alpha)\int_1f_j\frac{dz}{z}\right)\in\QB(f,z);$$ and $F_1\in\QQ(f,z)$.
\end{thm}

\begin{proof}
We consider the $\ZZ$-VHS $\Res_{\cO_K/\ZZ}\BV=:\BH\subset R^0\pi_*\ZZ_{\X}$, which has the property that $\BH\otimes_{\ZZ}\cO_K=\oplus_{j=1}^s\BV_j$, where $\BV_j$ ranges over the distinct conjugates $\BV_{\langle k_j\fa\rangle,\langle k_j\fb\rangle}$ and $s=[K\colon\QQ]$. The local exponents of the $j^{\text{th}}$ conjugate $\D$-module are\footnote{Note that $\sum_i(\langle k_j\fb_i\rangle-\langle k_j\fa_i\rangle)-1 > -1$ since $\langle k_j\fb_r\rangle=1$ and the $\langle k_j\fb\rangle,\,\langle k_j\fa\rangle$ are totally interlaced.} $$\textstyle\{1-\langle k_j\fb\rangle\},\;\{0,1,\ldots,r-2,\sum_i(\langle k_j\fb_i\rangle-\langle k_j\fa_i\rangle)-1\},\text{ and }\{\langle k_j\fa_i\rangle\};$$ and the last proof's logic gives that $\tilde{f}_j\frac{dz}{z}$ belongs to $\Gamma(\V_{j,e}'\otimes \Omega^1_{\PP^1}\langle 0\rangle)\hookrightarrow \Omega^1(\XB)\langle X_0\rangle$. Moreover, as $f$ is a local basis element of an $\cO_{\X}$-local system of solutions, we have that $\tilde{f}|_{X_0}$ defines a divisor $\SE_0$ (over $K$) with coefficients in $\cO_K$. Choosing a basis $\{\theta_{\ell}\}$ of $\cO_K$, we may write $\SE_0=:\sum_{\ell}\theta_{\ell}\VE_{0\ell}$ for some $\ZZ$-divisors $\VE_{0\ell}$ (defined over $K$).

Now the \emph{functions} $f_j$ are not Galois conjugate, but the \emph{monodromies} of the local systems they generate are.  Writing $\mathrm{Gal}(K/\QQ)=\{\sigma_j\}$, we have in this sense $\VV_j={}^{\sigma_j}\VV_1$, with invariant integral generator at $0$ sent to invariant integral generator at $0$. Thus we have $\SE_j:=\tilde{f}_j|_{X_0}=\Res_{X_0}(\tilde{f}_j\tfrac{dz}{z})=\sum_{\ell}\sigma_j(\theta_{\ell})\VE_{0\ell}$, which is not a divisor with $\QQ$-coefficients. So it cannot be written in the form $\mathrm{dlog}(\tilde{F}_j)$ with $\tilde{F}_j\in\CC(\XB)$, and $\exp(\int\tilde{f}_k\tfrac{dz}{z})$ is transcendental. On the other hand, if we set $\ff_{\alpha}:=\sum_j\sigma_j(\alpha)f_j$ for any $\alpha\in \cO_K$, $D_{\alpha}:=\sum_j\sigma_j(\alpha)\SE_j=\sum_{\ell}(\sum_j \sigma_j(\alpha\theta_{\ell}))\VE_{0\ell}$ is a $\ZZ$-divisor.

Moreover, we claim that $\ff_{\alpha}\frac{dz}{z}$ is an integral Hodge class. Indeed, by Euler-Poincar\'e $\IH^1(\GG_m,\VV)$ is a rank-one $\cO_K$-module of Hodge type $(-1,-1)$, so that modulo torsion $$\IH^1(\GG_m,\HH_{\ZZ}(1))=\Res_{\cO_K/\ZZ}\IH^1(\GG_m,\VV(1))\cong \cO_K(0)\cong \ZZ(0)^{\oplus s}.$$ Since $\tilde{\ff}_{\alpha}\frac{dz}{z}\in \Gamma(\H_e'\otimes \Omega^1_{\PP^1}\langle 0\rangle)\cong \IH^1(\PP^1{\setminus}\{0\},\HH_{\CC})$ has integer residues, the claim is proved as in Step 2 above.  Pulling \eqref{eq5.1} back by the inclusion of this Hodge class in $\IH^1(\GG_m,\HH(1))$ yields a normal function $\nu_{\alpha}$. As before, $\tilde{F}_{\alpha}=\exp(\nu_{\alpha})\in \CC(\XB)$ has $(\tilde{F}_{\alpha})=D_{\alpha}$, and is defined over $\QQ$.

Finally, notice that $$\textstyle {}^{\sigma}D_{\alpha}=\sum_{\ell}(\sum_j\sigma\sigma_j(\alpha\theta_{\ell}))\,{}^{\sigma}\VE_{0\ell}=\sum_{\ell}(\sum_j\sigma_j(\alpha\theta_{\ell}))\,{}^{\sigma}\VE_{0\ell}$$ need not equal $D_{\alpha}$, since $\mathrm{Gal}(K/\QQ)$ acts on the support of $\VE_{0\ell}$. More precisely, at the points of $\XB$ over $(y,z)=(\tilde{f},z)=(\beta,0)$ (with $\beta\in \cO_K$), the coeffcients of $\SE_j$ is $\sigma_j(\beta)$, where that of $D_{\alpha}$ is $\sum_j\sigma_j(\alpha.\sigma(\beta))$, which for general $\alpha$ is different. However, if $\alpha=1$ we see that $D_{\alpha}$ is invariant, and thus, arguing as before, that $\tilde{F}_1\in\QQ(\XB)$.
\end{proof}

\section{Vignette \#3: Quasi-periods}\label{S7}

The variation of mixed Hodge structure $\BE_{\nu}$ associated to a normal function $\nu$ over $\U\subset \PP^1$ has a dual extension
\begin{equation}\label{eq7.1}
0\to \BQ(p-\ell)\to \BE_{\nu}^{\vee}(-\ell)\to \BH\to 0
\end{equation}
using $\BH^{\vee}(-\ell)\cong\BH$. In effect, the quasi-period $\FV$ pairs a (Hodge) lift of $\mu$ to $F^0\E_{\nu}^{\vee}$ with a ($\QQ$-Betti) lift of $1$ to $\EE_{\nu}(p)$ in \eqref{eq2.4}, and is thus a period of the extension \eqref{eq7.1}. The periods of \eqref{eq2.4} itself are obtained by pairing a lift $\tilde{\nu}$ of the normal function to $\H$ aganst a ($\QQ$-Betti) section of $\HH$.

Differentially, the behavior of $\FV$ is described by Proposition \ref{pr2.10}. We continue with the assumptions imposed there, and assume for convenience that the Hodge numbers of $\BH$ are all $1$, that $\HH$ has MUM at $z=0$ and Picard-Lefschetz monodromy at $\Sigma\cap \GG_m$, and $\deg(L)=|\Sigma\cap \GG_m|$.\footnote{These additional assumptions imply that the Frobenius dual $L^{\dagger}$ is $L$ itself. See \cite[\S7]{Ke}. Note also that we can just take $R=D\frac{1}{g}$, as is done below.} If $R$ is an operator annihilating the polynomial $g$, then $RL$ annihilates $\FV$ and we may identify $\E_{\nu}^{\vee}\cong\D/\D(RL)$ as a $\D$-module. This suggests the Frobenius-dual identification $\E_{\nu}\cong \D/\D(LR^{\dagger})$, and we may wonder if there is some lift $\tilde{\nu}$ annihilated by $LR^{\dagger}$.

To put the question in context, recall that any two (local) lifts $\tilde{\nu}$ and $\tilde{\nu}'$ to $\H$ differ by a section $\gamma$ of the local system $\HH(p)$ and a section $\eta$ of $F^p\H$. Since the latter pairs trivially with $\mu$ by type, the two versions of $\FV$ differ by a period, which is killed by the Picard-Fuchs operator $L$. This makes $g:=L\FV$ independent of the lift. On the other hand, if we apply $D$ to the lifts locally, this kills $\gamma$ but sends $\eta$ to a section of $F^{p-1}\H$. So this section \emph{depends on} the lift, which we must somehow constrain to say anything about $D\tilde{\nu}$. According to \cite[Thm.~5.5]{GK}, the key is to choose $\tilde{\nu}$ so that the corresponding representative of the infinitesimal invariant $\delta\nu$ lies not just in $F^{p-1}\H$ but in the highest Hodge filtrand $F^{\ell}\H$:

\begin{prop}\label{pr7.2}
There exists locally everywhere a choice of lift $\tilde{\nu}$, unique up to $\HH(p)$, such that
\begin{equation}\label{eq7.3}
D\tilde{\nu}=\mathsf{Q}_0(2\pi\ay)^{\ell}g\mu
\end{equation}
where $g:=L\FV$ and $\mathsf{Q}_0\in \QQ^*$.
\end{prop}

Suppose $g=\mathsf{k}z^a$ for simplicity, and write $\omega=(2\pi\ay)^{\ell}\mu$ for the $\QQ$-de Rham form. Then $R=D \frac{1}{z^a}$ and $R^{\dagger}=\frac{1}{z^a}D$, so that we have dual mixed homogeneous Picard-Fuchs equations
\begin{equation}\label{eq7.4}
D\frac{1}{z^a}L\FV=0\;\;\;\text{and}\;\;\;L\frac{1}{z^a}D\tilde{\nu}=0
\end{equation}
if $\tilde{\nu}$ is taken as in Prop.~\ref{pr7.2}.

For the remainder of this section we concentrate on a very classical special case, where $\BH$ is the VHS on $H^1$ of a relatively minimal family $\{X_z\}$ of elliptic curves with section ($\ell=1$), and $\nu\in\ANF_{\PP^1}(\BH(1))$ ($p=1$) is the family of Abel-Jacobi maps arising from a divisor $\W$ on the total space $\XB$ whose restrictions $W_z:=\W\cdot X_z$ to smooth fibers has $\deg(W_z)=0$. (Note that $\mathsf{Q}_0^{-1}$ is the number of components of the semistable singular fiber over $z=0$.)

In order to force the existence of a unique normal function of this type we consider families with $\mathrm{ih}^1(\PP^1,\HH)=1$. Our options here are:
\begin{itemize}[leftmargin=0.9cm]
\item [(i)] $5$ semistable singular fibers;
\item [(ii)] $3$ semistable and $1$ non-semistable; or
\item [(iii)] $1$ semistable and $2$ non-semistable.
\end{itemize}
Now we also want to have $\deg(L)=2$ and $h=1$ in Prop.~\ref{pr2.10}, thereby forcing $g=\mathsf{k}z$ for some $\mathsf{k}\in \QQ^*$ ($a=1$ in \eqref{eq7.4}). This is incompatible with (i), whereas (iii) is incompatible with our other assumptions. Using Herfurtner's list \cite[Table 3]{He}, we find (up to quadratic twist) 17 configurations of Kodaira fiber types for (ii). Several of these are closely related to modular families (e.g.~by a twist), but --- since they have rank-one Mordell-Weil group --- cannot be exactly modular.

The 4 hypergeometric $\QQ$-VHS of weight 1 and rank 2 with $\fb=(1,1)$ have $\fa=(\tfrac{1}{2},\tfrac{1}{2})$, $(\tfrac{1}{3},\tfrac{2}{3})$, $(\tfrac{1}{4},\tfrac{3}{4})$, and $(\tfrac{1}{6},\tfrac{5}{6})$, and are realized by elliptic surfaces with singular fiber configurations $\RI_4\RI_1\RI_1^*$, $\RI_3\RI_1\RIV^*$, $\RI_2\RI_1\RIII^*$, and $\RI_1\RI_1\RII^*$ respectively. The latter three yield, by $2{:}1$ basechange and quadratic twist, 5 of the Herfurtner families. For instance, the double-cover of the $(\tfrac{1}{4},\tfrac{3}{4})$ family is a quadratic twist of the elliptic modular surface for $\Gamma_1(4)$, and the normal function it supports is described explicitly in \cite[\S10]{GK}. Here we concentrate on the $(\tfrac{1}{3},\tfrac{2}{3})$ family, and describe the nontorsion section (or normal function) supported on its basechange.

\begin{example}\label{ex7.5}
Let $\vf:=(x-y)(1+x^{-1})(1+y^{-1})$ with Newton polygon $\Delta$, $\BP:=\PP_{\Delta}$ the del Pezzo surface of degree 6, and $\BP\times \PP^1\supset \XB\overset{\pb}{\to}\PP^1_z$ the minimal smooth family obtained by compactifying level sets $\vf=z^{-1}$ to $X_z$. The singular fibers are $\RI_6$ ($z=0$), $2\,\RI_1$ ($z=\tfrac{\pm\ay}{3\sqrt{3}}$), and $\RIV$ ($z=\infty$).  Taking
$$\omega_z:=\Res_{X_z}\left(\frac{\frac{dx}{x}\wedge\frac{dy}{y}}{1-z\vf}\right)\in \Omega^1(X_z),$$ the period over a vanishing cycle $\alpha$ at $z=0$ is\footnote{The period over a complementary generator $\beta$ of $\HH_{\ZZ}^{\vee}$ can be computed by differentiating the \emph{Betti deformation} $\sum_{n\geq 0}\frac{\Gamma(n+s+\frac{1}{3})\Gamma(n+s+\frac{2}{3})}{\Gamma(\frac{1}{3})\Gamma(\frac{2}{3})\Gamma(n+s+1)^2}z^{n+s}$ with respect to $s$, setting $s=0$, and multiplying by $3$. See \cite[Appendix A]{Ke}.}
\begin{equation}\label{eq7.6}
\begin{split}
\Pi(z)&=\frac{1}{2\pi\ay}\int_{\alpha_z}\omega_z=\sum_{n\geq 0}z^n[\vf^n]_{\text{const.}}\\ &=\sum_{m\geq 0}\frac{3m!}{m!^3}(-z^2)^m=\,{}_2F_1\left(\substack{\frac{1}{3},\frac{2}{3}\\ 1}\middle| -27z^2 \right).
\end{split}
\end{equation}
Writing $\fz=-27z^2$, the Picard-Fuchs operator is therefore
\begin{equation}\label{eq7.7}
\begin{split}
L&=D_{\fz}^2-\fz(D_{\fz}+\tfrac{1}{3})(D_{\fz}+\tfrac{2}{3})\\
&=\tfrac{1}{4}\left\{D_z^2+27z^2(D_z+\tfrac{2}{3})(D_z+\tfrac{4}{3})\right\},
\end{split}
\end{equation}
which has degree 2 in $z$. There is an explicit Hauptmodul $\SH$ for $\Gamma_1(3)$ under which $\Pi_0(\fz):=\,{}_2F_1\left(\substack{\frac{1}{3},\frac{2}{3}\\1}\middle|\fz\right)$ pulls back to $A(q):=\Pi_0(\SH(q))\in M_1(\Gamma_1(3))$, see \cite[Ex.~10.10]{DK1}.

Now there are 6 ``constant'' sections, including $(x,y)=(-1,0)$, obtained by intersecting components of $\BP\setminus\GG_m^2$ with $X_z$.  Differences of these sections produce torsion normal functions.  Instead, consider the divisor $\W\in Z^1(\XB)$ with fibers 
\begin{equation}\label{eq7.8}
W_z=\W\cdot X_z=2\left[\left(\tfrac{1+z}{1-z},\tfrac{2z}{1-z}\right)\right]-2\left[\left(-1,0\right)\right],
\end{equation}
and let $\nu=\nu_{\W}\in \ANF_{\PP^1}(\BH(1))_{\ZZ}$. This is given locally by a difference $e_{\ZZ}-e_F$ as in \eqref{eq2.3}, where:
\begin{itemize}[leftmargin=0.5cm]
\item $e_{\ZZ}(z)$ is represented by $2\pi\ay\delta_{\Gamma_z}$, with $\Gamma_z$ a 1-chain bounding on $W_z$ and $\delta_{(\cdot)}$ the current of integration; and
\item	$e_F(z)$ is represented by a differential of the 3rd kind, such that taking $d$ as a current yields $2\pi\ay\delta_{W_z}$.
\end{itemize}
So we see that $\FV=\langle 2\pi\ay\delta_{\Gamma_z}-e_F(z),\tfrac{1}{2\pi\ay}\omega_z\rangle=\int_{\Gamma_z}\omega_z$ is just the classical Abel-Jacobi integral. By Proposition \ref{pr2.10}, we must have
\begin{equation}\label{eq7.9}
L\FV=\mathsf{k}z=\frac{\ay\mathsf{k}}{3\sqrt{3}}{\fz}^{\frac{1}{2}}
\end{equation}
for some $\mathsf{k}$; if $\mathsf{k}\neq 0$, then $\nu$ is nontorsion.  (This idea goes back at least to Manin's seminal paper \cite{Ma} on Mordell's conjecture.)

To compute $\mathsf{k}$, we compute a few terms of the power series for $\FV(z)$ about $z=0$. Let $\Omega_z=\frac{\frac{dx}{x}\wedge\frac{dy}{y}}{1-z\vf}$, $P_z$ be a path on $\CC\setminus\ay\RR_{\leq 0}$ from $-1$ to $\tfrac{1+z}{1-z}$, $C_z:=\Gamma_z\times\{|y|=\epsilon\}$ ($1\gg \epsilon>0$ fixed), and $\Gamma_z:=(P_z\times\{|y|<\epsilon\})\cap X_z$. Then for $|z|\ll \epsilon$, $\FV(z)$ is given by
\begin{equation}\label{eq7.10}
\begin{split}
2&\int_{\Gamma_z}\Res_{X_z}(\Omega_z)=\frac{2}{2\pi\ay}\int_{C_z}\Omega_z=2\int_{-1}^{\frac{1+z}{1-z}}\sum_k[\vf^k]_{y\text{-const}}z^k\frac{dx}{x}
\\	
&=2\int_{-1}^{\frac{1+z}{1-z}}\left(1+(x+x^{-1})z+(x^2-2x-6-2x^{-1}+x^{-2})z^2+\cdots\right)\frac{dx}{x}.
\end{split}
\end{equation}
Writing out terms of the integrand up to $z^5$, and expanding $\frac{1+z}{1-z}$ as a power-series in the result, \eqref{eq7.10} becomes
\begin{equation}\label{eq7.11}
-2\pi\ay+12z+12\pi\ay z^2-140z^3-180\pi\ay z^4+\tfrac{12012}{5}z^5+\cdots.
\end{equation}
We can modify $\Gamma_z$ by adding a copy of $\alpha_z$, which gives finally
\begin{equation}\label{eq7.12}
\FV(z)=12z-140z^3+\tfrac{12012}{5}z^5-\cdots.
\end{equation}
Applying $L$ we find that $\mathsf{k}=3$.

Having determined $\mathsf{k}$, there is a way to get \emph{all} the power-series coefficients of $\FV$. The Frobenius deformation \eqref{eq3.5} for the VHS in the hypergeometric parameter $\fz$ is $\Phi(s,\fz)=\sum_{n\geq 0}\frac{[s+\frac{1}{3}]_n[s+\frac{2}{3}]_n}{[s+1]_n^2}\fz^{s+n}$, and \eqref{eq3.6} gives $L\Phi(\tfrac{1}{2},\fz)=\fz^{\frac{1}{2}}/4$. Since there is a unique solution to \eqref{eq7.9} single-valued in $z$ about $0$, and vanishing at $0$, and \eqref{eq7.12} and $\frac{4\ay}{\sqrt{3}}\Phi(\frac{1}{2},\fz)$ both fit this description, we must have
\begin{equation}\label{eq7.13}
\begin{split}
\FV(z)&=\frac{\ay}{\sqrt{3}}\sum_{n\geq 0}\frac{[\frac{5}{6}]_n[\frac{7}{6}]_n}{[\frac{1}{2}]_{n+1}^2}\fz^{n+\frac{1}{2}}
\\	
&= 12\sum_{n\geq 0}(-1)^n\frac{(6n+1)!\,n!^3}{3n!2n!(2n+1)!^2}z^{2n+1},
\end{split}
\end{equation}
which matches the computed terms of \eqref{eq7.12}.

Finally, the lift $\tilde{\nu}$ about $z=0$ corresponding to this choice of $\FV$ must satisfy \eqref{eq7.3} in the form ($\mathsf{Q}_0=\frac{1}{6}$, $g=3z$)
\begin{equation}\label{eq7.14}
D_z\tilde{v}=\frac{1}{2}z\omega\;\;\iff\;\;\tilde{\nu}=\frac{1}{2}\int_0\omega dz.
\end{equation}
Then $\langle\alpha,\tilde{\nu}\rangle$ and $\langle\beta,\tilde{\nu}\rangle$ yield extension periods for $\BE_{\nu}$, for instance
\begin{equation}\label{eq7.15}
\langle\alpha,\tilde{\nu}\rangle=\frac{2\pi\ay}{2}\int_0\Pi(z)dz=\pi\ay\sum_{m\geq0}\frac{(-1)^m}{2m+1}\frac{3m!}{m!^3}z^{2m+1}.
\end{equation}
Expressions like this are needed for computing regulators of Hadamard product cycles as in \cite[\S10]{GK}, and are relevant to computing Hodge-theoretic height pairings in families.\hfill$\square$
\end{example}

In the remainder of this section we explain how to write down a 1-current on the total space representing the lift $\tilde{\nu}$ in Proposition \ref{pr7.2}. This is a \emph{local} computation, and involves making the $e_F$ part of the lift ``formally explicit''. So we can get by with using the period ratio $\tau\in \mathfrak{H}$ and $q=e^{2\pi\ay\tau}$ as local base-parameters. We write $E_q=\CC^*/q^{\ZZ}\cong\CC/\ZZ\langle 1,\tau\rangle$ for the elliptic fibers, with coordinate $t=e^{2\pi\ay u}$. The divisor $\W$ is locally represented by $\sum_jm_j[t_j]$, with $\sum_jm_j=0$, and $t_j$ regarded as functions of $q$ (valued in $\CC^*$, not just $E_q$). Denote by $\E\overset{\pi}{\to}\mathcal{B}$ the total space over a simply connected subset of the punctured disk $0<|q|<1$. Computations are straightforward and mostly left to the reader.

The form
\begin{equation}\label{eq7.16}
\hat{\omega}(t):=\mathrm{dlog}(1-t)+\sum_{n>0}\mathrm{dlog}\{(1-q^nt)(1-q^nt^{-1})\}
\end{equation}
is not well-defined on $E_q$ since $\hat{\omega}(t/q)-\hat{\omega}(t)=\mathrm{dlog}(-t/q)$. The combination
\begin{equation}\label{eq7.17}
\textstyle\hat{\omega}_{\W}(t):=\sum_j m_j\,\hat{\omega}(t/t_j)
\end{equation}
becomes well-defined on $E_q$ (since $\sum m_j=0$), but on $\E$ it is not since
\begin{align*}
\hat{\omega}_{\W}(t/q)-\hat{\omega}_{\W}(t)&\textstyle=-\sum_j m_j\mathrm{dlog}(t_j)\\&\textstyle=-\left(D_q\sum_j m_j\log(t_j)\right)\frac{dq}{q}=:-f(q)\frac{dq}{q}.
\end{align*}
Replacing $\hat{\omega}_{\W}$ by
\begin{equation}\label{eq7.18}
\tilde{\omega}_{\W}:=\hat{\omega}_{\W}(t)-\frac{\log|t|}{\log|q|}\sum_j m_j\,\mathrm{dlog}(t_j)
\end{equation}
corrects this: $\tilde{\omega}_{\W}$ is a well-defined form in $A^1(\E)\langle \W\rangle$. However, it has the defect that, as a current, 
\begin{equation*}
d[\tilde{\omega}_{\W}]\equiv 2\pi\ay\delta_{\W}-\mathrm{dlog}|t|\wedge\frac{f(q)}{\log|q|}\frac{dq}{q}
\end{equation*}
up to the image of $\pi^*$. We want the second term to be holomorphic in the fiberwise ($t$-)direction.

Consider the lift of $\mathrm{dlog}(t)=2\pi\ay\,du$ (on fibers) to
\begin{equation*}
\widetilde{\mathrm{dlog}(t)}=\mathrm{dlog}(t)-\frac{\log|t|}{\log|q|}\mathrm{dlog}(q)
\end{equation*}
on $\E$. It has $d[\widetilde{\mathrm{dlog}(t)}]=-\frac{\mathrm{dlog}|t|}{\log|q|}\wedge\mathrm{dlog}(q)$, and so 
\begin{equation}\label{eq7.19}
\Omega_{\W}:=\tilde{\omega}_{\W}-f(q)\,\widetilde{\mathrm{dlog}(t)}
\end{equation}
satisfies $d[\Omega_{\W}]=2\pi\ay\delta_{\W}+(D_qf)\mathrm{dlog}(t)\wedge\mathrm{dlog}(q)$.

Taking $\Gamma$ a 3-chain on $\E$ bounding on $\W$, we may interpret 
\begin{equation}\label{eq7.20}
\hat{\nu}:=\hat{\omega}_{\W}-f(q)\,\mathrm{dlog}(t)-2\pi\ay\delta_{\Gamma}
\end{equation}
as a relative ($d_{\text{rel}}$-closed) 1-current on fibers. Let $\tilde{\nu}$ be its class in $\H$. To compute $\delta\nu=\nabla\tilde{\nu}$, we apply $d$ to a lift of $\hat{\nu}$ on $\E$: namely, to $\Omega_{\W}-2\pi\ay\delta_{\Gamma}$. This yields $(D_qf)\,\mathrm{dlog}(t)\wedge\mathrm{dlog}(q)$, and so we conclude that
\begin{equation}\label{eq7.21}
D_q\tilde{\nu}=(D_q f)\,\mathrm{dlog}(t)
\end{equation}
lies in $F^1\H^1$. We have proved

\begin{prop}\label{pr7.22}
The correct choice of lift in Proposition \ref{pr7.2} is given by taking $e_F$ to be the Green's current $\hat{\omega}_{\W}-f(q)\frac{dt}{t}$ on each fiber.
\end{prop}

To bring this full circle to \eqref{eq7.3}, if $\omega=F(q)\,\mathrm{dlog}(t)$, and $z=Z(q)$, with $F(q)=\frac{1}{2\pi\ay}\Pi(Z(q))$, then $D_z\tilde{\nu}=\frac{D_zf}{\Pi}\omega$. Thus $D_zf=cg\Pi$ for some constant $c$, and $f=\int cg\Pi\frac{dz}{z}$. This looks similar to \eqref{eq7.15}, and sure enough one computes directly or reasons from \eqref{eq7.21} that $\int_{\alpha}\hat{\nu}=2\pi\ay f$.

Something particularly intriguing happens in the scenario of Example \ref{ex7.5}. Writing as before $\fz=\SH(q)$ as a Hauptmodul and $A(q)=\Pi_0(\SH(q))$, we have for some constant C
\begin{equation*}
D_{\fz}\tilde{\nu}=C\fz^{\frac{1}{2}}\omega\;\;\;\implies\;\;\;D_q\tilde{\nu}=C\frac{D_q\SH}{\sqrt{\SH}}A\frac{dt}{t}.
\end{equation*}
Hence $D_q\langle \tilde{\nu},du\rangle=\frac{1}{2\pi\ay}\langle\tilde{\nu},\alpha\rangle$ and
\begin{equation*}
D_q^2\langle\tilde{\nu},du\rangle=\tfrac{1}{2\pi\ay}D_q\langle\tilde{\nu},\alpha\rangle=\tfrac{1}{2\pi\ay}\langle D_q\tilde{\nu},\alpha\rangle
\end{equation*}
is on one hand $\frac{1}{2\pi\ay}\langle(D_qf)\frac{dt}{t},\alpha\rangle=D_q f$ and on the other $C\frac{D_q \SH}{\SH}A\cdot\sqrt{\SH}$, 
which is a weight-3 modular form for $\Gamma_1(3)$ times the ``quadratic twist'' $\sqrt{\SH}$. This makes $\langle\tilde{\nu},du\rangle$ a \emph{twisted Eichler integral}, and the quasi-period
\begin{equation*}
\FV=\langle\tilde{\nu},\tfrac{1}{2\pi\ay}\omega\rangle=A\langle\tilde{\nu},du\rangle
\end{equation*}
its product with a weight-1 modular form. Though $\FV$ is merely a classical Abel-Jacobi map, this nicely parallels the analysis of higher normal functions and Feynman integrals in \cite{DK1} and \cite{BKV}.

\curraddr{${}$\\
\noun{Department of Mathematics}\\
\noun{Washington University in St. Louis}\\
\noun{St. Louis, MO} \noun{63130, USA}}
\email{${}$\\
\emph{e-mail}: matkerr@wustl.edu}

\end{document}